\documentclass{article}
\usepackage{graphicx} 
\usepackage{amsmath, amssymb, amsthm}
\usepackage{graphicx}
\usepackage{bm}
\usepackage{upgreek}
\usepackage{enumitem} 
\usepackage{hyperref}
\hypersetup{colorlinks,linkcolor={blue},citecolor={red},urlcolor={blue}}
\usepackage{cleveref}
\usepackage{mathtools}
\usepackage{enumitem}

\newtheorem*{assumptionA*}{\textbf{Assumptions}\hspace{-3pt}}
\Crefname{assumptionA}{\textbf{Assumptions}\hspace{-3pt}}{\textbf{H}\hspace{-3pt}}
\crefname{assumptionA}{\textbf{Assumptions}}{\textbf{A}}

\newtheorem{thm}{Theorem}

\newtheorem{lemma}{Lemma}
\newtheorem{corollary}{Corollary}
\newtheorem{prop}{Proposition}

\newtheorem{note}{Note}
\newtheorem{rem}{Remark}
\usepackage[letterpaper,top=2cm,bottom=2cm,left=4cm,right=4cm,marginparwidth=1.75cm]{geometry}

\usepackage{graphicx}
\graphicspath{ {./Diagrams/} }
\usepackage{float}

\title{A Tamed Euler Scheme for SDEs with Non-Locally Integrable Drift Coefficient}
\author{Tim Johnston, Sotirios Sabanis}
\date{April 2024}

\begin{document}

\maketitle
\begin{abstract}
  In this article we show that for SDEs with a drift coefficient that is non-locally integrable, one may define a tamed Euler scheme that converges in $L^p$ at rate $1/2$ to the true solution. The taming is required in this case since one cannot expect the regular Euler scheme to have finite moments in $L^p$. Our novel proof strategy involves controlling the inverse moments of the distance of scheme and the true solution to the singularity set.  We additionally show that our setting applies to the case of two scalar valued particles with singular interaction kernel. To the best of the author's knowledge, this is the first work we are aware of to prove strong convergence of an Euler-type scheme in the case of non-locally integrable drift.
\end{abstract}





\section{Introduction}

In this article we consider SDEs taking values in an open set $D\subset \mathbb{R}^d$ with additive noise, and the performance of the associated Euler scheme. That is, for drift coefficient $b:D\to\mathbb{R}^d$ we consider
\begin{equation}\label{eq: SDE main}
    dX_t=b(X_t)dt+dW_t,\;\;\;t\geq 0,\;\;\;X_0=x_0\in D,
\end{equation}
and the Euler scheme approximation
\begin{equation}\label{eq: Euler Scheme}
    dX^n_t=b(X_{\kappa_n(t)})dt+dW_t,\;\;\;t\geq 0,\;\;\;X^n_0=x_0\in D, 
\end{equation}
where $\kappa_n(t):=\lfloor nt\rfloor/n$ is the projection backwards onto the grid $\{0,1/n,2/n,...\}$. The performance of the Euler scheme has been extensively studied under a wide variety of assumptions on the coefficients. It is well known that the Euler scheme converges in $L^p$ at rate $1/2$ given Lipschitz drift and diffusion coefficient, see \cite{kloeden2013numerical}, and in an almost-sure sense under significantly weaker assumptions, see \cite{Gyongy1998}. However the Euler scheme fails to converge at all in $L^p$ (and in fact diverges) when the coefficients are superlinear, see \cite{articlediv}, and as a result in order to simulate such an SDE explicitly the notion of `tamed' schemes was first introduced in \cite{da2b46b6-84fa-3eed-be6d-20c143c60df1} and developed independently in \cite{articletamed}. For these schemes, one replaces the drift coefficient $b$ in the Euler scheme with functions $b_n$ that depend on the stepsize. For more information on the topic of tamed schemes see \cite{articledif, arnulf2012, ChamanKumar2017, 10.1214/19-AOP1345} and references therein. Furthermore, much research in recent years has focused on the perfomance of the Euler scheme for discontinuous coefficients. In this case, the Euler scheme generally converges well. Indeed, for bounded and measurable drift coefficient and uniformly elliptic diffusion coefficient the Euler scheme converges in $L^p$ at rate $1/2-$, see \cite{10.1214/22-AAP1867}. Similar results have been shown in \cite{articlemg, MullerGronbach2022, 10.1093/imanum/draa078, 10.1214/20-EJP479}. For a survey of this fast moving area of research one could consult \cite{MULLERGRONBACH2024101870}.

In this article we consider the performance of the Euler scheme for another class of non-Lipschitz drift coefficients, namely coefficients that are locally unbounded, in particular non-locally integrable. In \cite{ojm/1396966224, ZHANG2006447, Ding2021} strong convergence of the Euler scheme in $L^p$ at rate $1/2$ is proven for drifts with logarithmic, and therefore integrable, poles. These results therefore match the convergence results achieved in this article for drifts with more severe poles. The weak convergence of a similar scheme to the one considered in this article was considered in the case where the drift coefficient obeyed a certain intergrability property in \cite{10.1214/23-AAP2006}, however the necessity of such a property ruled out applicability to many SDEs of interest. Additionally, in \cite{journel2024weak}, a similar scheme was considered for the underdamped Langevin dynamics for Lennard-Jones interactions, that is, for a second order SDE whose drift coefficient is the gradient of a very singular scalar potential, and weak convergence was proven. Furthermore, in \cite{1b199059eaf9494c842eec278d3e7aac} a tamed Euler scheme similar to those discussed previously is shown to converge in probability under a setting similar to that considered in this article, however no rate is presented. In this article we present another tamed scheme, and show that one can obtain strong convergence at rate $1/2$ even when the drift coefficient $b$ is not-locally integrable, given certain conditions that ensure the true solution is `pushed' away from points where the drift is singular. Key to our proofs is a certain stopping time argument which means we do not have to consider what happens when either the scheme or the true solution get too close to `bad' points. A stopped scheme is used under different conditions to approximate the exit time of a stochastic process in \cite{GOBET2000167}, and the related idea of `discarding' bad trajectories is shown to ensure weak convergence up to some $\epsilon>0$ in a variety of non-Lipschitz settings in \cite{5aa5837d-f9a1-33da-b667-726b0d620667}. However, this is the first paper we are aware of to study strong convergence of an Euler-type scheme for non-locally integrable coefficients. Strong convergence of numerical methods is important for obtaining bounds for sampling algorithms based on discretisations of SDEs in Wasserstein distance, see \cite{DALALYAN20195278, 3b99fee3596546ae809bc62b56a361be} and many others, as well as for multi-level Monte Carlo methods, see \cite{Giles_2015}.

For the study of SDEs with singular drift, it is important to show that the original continuous time dynamics are well-posed, that is, the SDE has a unique strong solution. If this is not the case then one cannot expect to construct a solution to the dynamics from the driving noise alone, and as a result one cannot expect any numerical method to converge strongly. In \cite{1b199059eaf9494c842eec278d3e7aac} strong solutions are constructed given more general conditions than those presented in this article, using directly an Euler-type scheme as pioneered in the classic articles \cite{113194bd7a444bad9e31624261c18d4b, Gyongy1998}. The literature on existence and uniqueness of strong solutions for SDEs with singular drift is vast, see for instance \cite{ZHANG20051805, Gyongy2001, rockner2021sdescriticaltimedependent}, as well as the book-length study \cite{cherny2004singular}. In this article we choose to prove existence and uniqueness directly, since our argument relates fundamentally to the key strategy of the paper, which is to control the inverse moments of the distance of the scheme and the true solution to the singular points of the drift. We observe that the results of \cite{1b199059eaf9494c842eec278d3e7aac} could be applied to this end by choosing an appropriate Lyapunov function $V$, for instance $V=\rho^{-2}$, for $\rho$ as given in the assumptions below.

The primary motivation for the study of such singular drift coefficients comes from particle dynamics, where one considers particle systems with strong repulsive forces at small distance, for instance the famous Lennard-Jones dynamics, wherein particles at distance of $r$ away from one another repulse each other with force $\sim r^{-12}-r^{-6}$. In Corollary \ref{cor: interacting particles} we show that the scheme presented is this article converges strongly to the true solution of this dynamics in the `overdamped case' where friction does not play a role, for two scalar valued particles. In future we would like to refine this restriction to more general classes of interacting particle systems with singular interaction.

One cannot expect the Euler scheme to converge for coefficients that are not locally integrable. Indeed, if $b$ is not in $L^1_{loc}(D)$ then there exists an open set $D'\subset D$ with compact closure such that $\int_{D'}\lvert b \rvert =\infty$. Therefore for any Gaussian random variable $Z$ that is absolutely continuous with respect to the Lebesgue density one has $E\lvert b(Z)\rvert=\infty$, due to the lower bounded on the density of a Gaussian on any compact set. It can easily be deduced from this that \eqref{eq: Euler Scheme} satisfies
\begin{equation}
    E\lvert X^n_t\rvert=\infty,
\end{equation}
for every $t>0$. Therefore, we consider a tamed Euler scheme whose drift coefficient is bounded for every $n\geq 1$. This, in combination with a certain stopping time argument, allows us to demonstrate strong convergence to the true solution. In particular, our argument involves controlling both the regular moments, as well as the inverse moments to the singularity set, of the scheme and the true solution. Furthermore, in Corollary \ref{cor: interacting particles} we demonstrate that the conditions of Theorem \ref{thm: main theorem} are general enough to contain a large class of singular potentials for interacting particle systems, including the famous Lennard-Jones interaction in the case of two scalar particles.

\subsection{Notation}
We conclude this section by introducing the notation we shall use in this article. We use upper subscripts to denote components of elements of $\mathbb{R}^d$, so that $x^j$ for instance denotes the $j$th component of $x\in \mathbb{R}^d$. For $A\subset \mathbb{R}^d$ let $C^r(A)$ denote $r$-times continuously differentiable functions from $A$ to $ \mathbb{R}$, and likewise let $C^r(A, \mathbb{R}^m)$ denote $r$-times continuously differentiable functions from $A$ to $\mathbb{R}^m$. Furthermore, we denote by $C^r_c(X)$ and $C^r_c(X, Y)$ for the subset of the previous function spaces containing functions with compact support. Let $\lvert \cdot \rvert$ denote the Euclidean norm for elements of $\mathbb{R}^d$ as well as the Frobenius norm for matrix elements of $\mathbb{R}^{d \times d}$. For the inner product on $\mathbb{R}^d$ and the Frobenius inner product on $\mathbb{R}^{d\times d}$, we write $\langle x, y \rangle$. For $R>0$ let $B_R \subset \mathbb{R}^d$ denote the ball of radius $R$ around $0$, and for $x\in \mathbb{R}^d$ let $B_R (x)\subset \mathbb{R}^d$ denote the ball of radius $R$ around $x$. For a subset $A \subset \mathbb{R}^d$ define $dist(x, A):= \inf \{ \lvert x- y \rvert  \; \vert y \in A \}$, the boundary of $A$ as $\partial A$, the closure $\bar{A}$ and the interior $int(A)$. Finally if $A,B \subset \mathbb{R}^d$, we define the sum $A+B:=\{a+b \; \vert a\in A, b\in B\}$.

\section{Assumptions and Main Theorem}\label{sec: assump and main thm}

Our assumptions and main result are given below. We consider the case where $b$ and other key functions given by the singularity set obey an `inverse polynomial Lipschitz condition', meaning the Lipschitz constant can blow up as a negative power of the distance to the singularity set. We additionally assume that the drift coefficient `pushes the solution away' from the points where $b$ blows up. In the case of two interacting particle systems with singular interactions $S$ is the set of points for which two particles meet, as demonstrated in Corollary \ref{cor: interacting particles}.
\begin{assumptionA*}
There exists an open set $D\subset \mathbb{R}^d$ and a measurable set $S \subset \mathbb{R}^d$ such that the initial condition $x_0$ of \eqref{eq: SDE main} satisfies $x_0\in D$, and additionally $D\cap S=\emptyset$ and $\partial D\subset S$. We assume the function $\rho:D\to[0,\infty)$, given by
\begin{equation}\label{eq: rho defn}
    \rho(x):=dist(x,S),
\end{equation}
is twice continuously differentiable on $D$, that is, $\rho\in C^2(D)$. Finally $b:\mathbb{R}^d\to\mathbb{R}^d$ is a measurable function for which there exists $c,\alpha,l, h_1, h_2, h_3, h_4>0$ and $1<\beta<\alpha$ such that for every $x, y\in D$
\begin{equation}\label{assmp: monotonicity}
    \langle b(x)-b(y), x-y \rangle\leq c\lvert x-y \rvert^2,
\end{equation}
\begin{equation}\label{assmp: rho pushes away}
    \langle \nabla \rho(x), b(x)\rangle \geq h_1 \rho(x)^{-\alpha}-h_2\rho(x)^{-\beta}-h_3-h_4\rho(x),
\end{equation}
\begin{equation}\label{assmp: second derivative rho}
     \lvert \nabla ^2\rho(x)\rvert\leq c\biggr (1+ \rho(x)^{-\beta}\biggr),
\end{equation}
\begin{equation}\label{assmp: inverse poly Lipschitz b}
      \lvert  b(x)- b(y)\rvert \leq c\biggr (1+ \rho(x)^{-l}+ \rho(y)^{-l}\biggr)\lvert x-y \rvert.
\end{equation}
\end{assumptionA*}
   The key assumption is \eqref{assmp: rho pushes away}, which demonstrates that the dynamics push the true solution away from the singular points of the drift, which lie on $S$. Actually \eqref{assmp: rho pushes away} furthermore implies that for every sequence $x_n \in D$ such that $x_n\to s \in S$, one has $b(x_n)\to \infty$. 
\begin{rem}
    For a simple case obeying the above assumptions, for any $\alpha>1$ one could take $d=1$, $D=(0,\infty)$, $b(x)=x^{-\alpha}$, $S=\{0\}$, $\rho(x)=x$, $\beta=0$ and $l=\alpha+1$. In particular, since the derivatives of $\rho$ are constant, this setting simplifies certain calculations. One observes that \eqref{assmp: monotonicity} follows since $b$ is decreasing, and \eqref{assmp: inverse poly Lipschitz b} follows by the mean value theorem.
\end{rem}

    The monotonicity assumption \eqref{assmp: monotonicity} as well as assumption \eqref{assmp: rho pushes away} mean that our assumptions do not cover interacting particles with singular interaction kernels in higher dimensions, nor with more than two particles. In Corollary \ref{cor: interacting particles} we show that our assumptions do indeed cover Lennard-Jones interactions for two particles in the scalar valued case where $d=1$. In future work we would like to lessen these assumptions so as to allow for more general examples.

\begin{thm}\label{thm: main theorem}
    Let the assumptions presented in this section hold for \eqref{eq: SDE main}. Let us fix $\delta>0$ and $w\in [0,\frac{1}{3l}]$. For $\rho$ as given in \eqref{eq: rho defn} and for all $n \in \mathbb{N}$, let us define the truncated coefficients $b_n:\mathbb{R}^d\to\mathbb{R}^d$ as
    \begin{equation}
        b_n(x)=b(x)1_{\{x\in D\} \cap\{\rho(x)\geq \delta n^{-w}\}}.
    \end{equation}
   Then if one defines $\kappa_n(t):=\lfloor nt\rfloor /n$ to be the projection backwards onto $\{0,1/n,2/n,...\}$, one has that \eqref{eq: SDE main} has a unique strong solution $(X_t)_{t\geq 0}$ and 
   \begin{equation}
       dX^n_t=b_n(X^n_{\kappa_n(t)})+dW_t
   \end{equation}
   is uniquely well defined for all $t\geq 0$. Furthermore, for any $T>0$ and $p>0$ there exists $c>0$ independent of $n\geq 1$ such that
    \begin{equation}\label{eq: conclusion Theorem 1}
        E\sup_{t\in [0,T]}\lvert X^n_t-X_t\rvert^p\leq cn^{-p/2}.
    \end{equation}
\end{thm}

\section{Continuous Time Solution}
In this section we prove that, given the assumptions in Section \ref{sec: assump and main thm}, there exists a unique strong solution of \eqref{eq: SDE main} that stays in $D$ almost surely. We begin by showing \eqref{eq: SDE main} has a unique strong solution up to some stopping time, and in Lemma \ref{lemma: true solution cont for all time} refine this to show \eqref{eq: SDE main} has a unique strong solution for all time. For our calculations we use a local version of It\^{o}'s formula presented in the appendix. Henceforth, we assume that the assumptions given in Section \ref{sec: assump and main thm} hold.

\begin{note}
    The constant $c>0$ appearing in the following Lemmas may depend on $T>0$ but never on $n\geq 1$, $\epsilon>0$ or any time parameters $t,s, u \geq 0$. 
\end{note}

\begin{lemma}\label{lemma: local definition continuous time}
 For any $\epsilon>0$ there exists a stopping time $\tau_\epsilon$ such that \eqref{eq: SDE main} has a unique strong $X_t$ solution on $[0,\tau_\epsilon]$, and furthermore
   \begin{equation}\label{eq: tau_epsilon definition}
       \tau_\epsilon=\inf\{t\geq 0\vert \; \rho(X_t)\leq \epsilon  \}.
   \end{equation}
\end{lemma}
\begin{proof}
    Fix $y_0\in D$. Since $D\subset\mathbb{R}^d$ is an open set, there exists an $a>0$ for which $B_a(y)\subset D$, and so by the definition of $\rho$ in \eqref{eq: rho defn} one has $\rho(y_0)>a>0$. Therefore by \eqref{assmp: inverse poly Lipschitz b}, for all $x\in D$ one may obtain the following bound
    \begin{align}\label{eq: upper bound for b}
       \lvert b(x)\rvert &\leq \lvert b(y_0)\rvert+\lvert b(x)-b(y_0)\rvert \nonumber \\
       &\leq c(1+\rho(x)^{-l})(1+\lvert x \rvert),
    \end{align}
    where the constant $c>0$ is independent of $x\in D$. Now let us define
    \begin{align}\label{eq: K epsilon defn}
    K_\epsilon&:=\{x\in D\vert\; \rho(x)> \epsilon  \},
    \end{align}
   so that we may define the truncated coefficients $b_\epsilon:\mathbb{R}^d\to\mathbb{R}^d$ by
    \begin{equation}
   \bar{b}_\epsilon(x):=b(x)1_{  K_\epsilon}.
    \end{equation}
   In particular, it follows that by setting $\epsilon = \delta n^{-w}$ one has $\bar{b}_\epsilon=b_n$. Furthermore, by \eqref{eq: upper bound for b} one observes that there exists a $c_\epsilon >0$ independent of $x$ for which $ \lvert\bar{b}_\epsilon(x)\rvert\leq c_\epsilon (1+\lvert x \rvert)$. Therefore, using Theorem 1.1 in \cite{ZHANG20051805} one may conclude that the SDE
  \begin{equation}\label{eq: epsilon truncated SDE}
    d\bar{X}^\epsilon_t=b_\epsilon(\bar{X}^\epsilon_t)dt+dW_t,\;\;\; \bar{X}^\epsilon_0=x_0,
  \end{equation}
  has a unique strong solution on $[0,\infty)$. Here $x_0\in D$ is the same initial condition as \eqref{eq: SDE main}. Let us furthermore define the stopping time  $\tau_\epsilon$ by
     \begin{equation}
       \tau_\epsilon:=\inf\{t\geq 0\vert \; \bar{X}^\epsilon_t\not \in K_\epsilon\}.
   \end{equation}
One observes then that 
\begin{equation}
d\bar{X}^\epsilon_t=b(\bar{X}^\epsilon_t)dt+dW_t,\;\;\;t\in [0,\tau_\epsilon) ,
\end{equation}
and therefore $X_t:=\bar{X}^\epsilon_t$ is a strong solution of \eqref{eq: SDE main} on $[0,\tau_\epsilon]$. Furthermore uniqueness follows since if this solution were non-unique on $[0,\tau_\epsilon]$ one could find another solution of \eqref{eq: epsilon truncated SDE} on $[0,\tau_\epsilon]$. Finally \eqref{eq: tau_epsilon definition} follows trivially since $X_t=\bar{X}^\epsilon_t$ on $[0,\tau_\epsilon]$.
\end{proof}
\begin{lemma}\label{lemma: strong cont moment bounds}
  Let $\tau_\epsilon$ be as in Lemma \ref{lemma: local definition continuous time}. Then for every $p>0$ and $T>0$ one has 
       \begin{align}\label{eq: strong cont bound}
    \sup_{\epsilon>0} E\sup_{t\in [0,T]}\rho(X_{t\wedge \tau_\epsilon})^{-p}<\infty.
    \end{align}
\end{lemma}
\begin{proof}
    Let $X_t$ be the solution to \eqref{eq: SDE main} on $[0,\tau_\epsilon]$ which is given in Lemma \ref{lemma: local definition continuous time}. Then by Proposition \ref{prop: local Ito} (see Appendix) one may write
    \begin{align}\label{eq: cont time negative moments}
    \rho(X_{t\wedge \tau_\epsilon})^{-p}&=\rho(x_0)^{-p}-p\int^{t\wedge \tau_\epsilon}_0\rho(X_s)^{-p-1}\langle \nabla \rho(X_s), b(X_s) \rangle ds\nonumber \\
    &-\frac{p}{2} \int^{t\wedge \tau_\epsilon}_0 \rho(X_s)^{-p-1}tr(\nabla ^2\rho(X_s))ds \nonumber \\
    &+\frac{p(p+1)}{2}\int^{t\wedge \tau_\epsilon}_0\rho(X_s)^{-p-2}tr(\nabla \rho(X_s)\otimes \nabla \rho(X_s))   ds\nonumber \\
    &-p\int^{t\wedge \tau_\epsilon}_0\rho(X_s)^{-p-1}\langle \nabla \rho(X_s),dW_s \rangle .
    \end{align}
    Now observe that for every $x,y\in D$ one has
    \begin{equation}\label{eq: Lipschitz calc}
        \rho(x)=\inf\{\lvert x-z\rvert \vert z\in S\}\leq \inf\{\lvert x-y\rvert +\lvert y-z\rvert\vert z\in S\}=\lvert x-y\rvert+\rho(y),
    \end{equation}
    so $\rho$ is Lipschitz with constant $1$. As a result, since by assumption $\rho\in C^2(D)$, one has $\lvert \nabla \rho \rvert\leq 1$, so by this and \eqref{assmp: rho pushes away} and \eqref{assmp: second derivative rho} one therefore sees that there exists $c_1,c_2, c_3, c_4, c_5>0$ not depending on $\epsilon>0$ or $t>0$ for which
       \begin{align}\label{eq: Ito cont with bound applied}
    \rho(X_{t\wedge \tau_\epsilon})^{-p}&\leq \rho(x_0)^{-p}+\int^{t\wedge \tau_\epsilon}_0(-c_1\rho(X_s)^{-p-1-\alpha}+c_2\rho(X_s)^{-p-1-\beta}) ds\nonumber \\
    &+\int^{t\wedge \tau_\epsilon}_0 (c_3\rho(X_s)^{-p-1} +c_4\rho(X_s)^{-p-2} )ds\nonumber \\
    &+\int^{t\wedge \tau_\epsilon}_0 c_5 \rho(X_s)^{-p} ds-p\int^{t\wedge \tau_\epsilon}_0\rho(X_s)^{-p-1}\langle \nabla \rho(X_s),dW_s \rangle .
    \end{align}
    Now observe that since $1<\beta <\alpha$ one has
    \begin{equation}\label{eq: bound on polynomial}
      \sup_{x\geq 0} (-c_1x^{-p-1-\alpha}+c_2x^{-p-1-\beta}+c_3 x^{-p-1}+c_4 x^{-p-2})<\infty.
    \end{equation}
    since the first term dominates in the limit as $x\to 0$. Therefore there exists a constant $c>0$ independent of $\epsilon>0$ or $t>0$ such that
    \begin{align}\label{eq: Ito cont with bound applied 2}
    \rho(X_{t\wedge \tau_\epsilon})^{-p}&\leq \rho(x_0)^{-p}+ct+\int^{t\wedge \tau_\epsilon}_0 c_5 \rho(X_{s\wedge \tau_\epsilon})^{-p} ds \nonumber \\
    &-p\int^{t\wedge \tau_\epsilon}_0\rho(X_s)^{-p-1}\langle \nabla \rho(X_s),dW_s \rangle .
    \end{align}
    Furthermore, since the integrand of the stochastic integral in \eqref{eq: Ito cont with bound applied 2} is bounded by the definition of $\tau_\epsilon$, it is a true martingale. As a result, by Grönwall's inequality one has for every $p>0$ that
        \begin{align}\label{eq: weak cont rho moment bound}
    \sup_{\epsilon>0}\sup_{t\in [0,T]}E\rho(X_{t\wedge \tau_\epsilon})^{-p}<\infty.
    \end{align}
    Now  to prove strong bounds on $[0,T]$, one may proceed by squaring \eqref{eq: Ito cont with bound applied 2}, taking supremum, and then expectation and the Burkholder-Davis-Gundy inequality, to obtain
          \begin{align}
    E\sup_{u\in [0,t]}&\rho(X_{u\wedge \tau_\epsilon})^{-2p}\leq c(1+t^2)+c\int^{t\wedge \tau_\epsilon}_0 E \rho(X_s)^{-2p} ds \nonumber \\
    &+cE\sup_{s\in [0,t]}\biggr (\int^{s\wedge \tau_\epsilon}_0\rho(X_s)^{-p-1}\langle \nabla \rho(X_s),dW_s \rangle \biggr)^2 \nonumber \\
    &\leq c\biggr(1+t^2+c\int^t_0E\rho(X_{s\wedge \tau_\epsilon})^{-2p-2}ds\biggr).
    \end{align}
Since \eqref{eq: weak cont rho moment bound} holds for all $p>0$, the result follows.
    \end{proof}
\begin{lemma}\label{lemma: true solution cont for all time}
   The SDE \eqref{eq: SDE main} has a unique strong solution on $[0,\infty)$ and
        \begin{equation}\label{eq: cont sol stays in D}
        P(X_t \in D\; \text{for every}\; t\geq 0)=1.
    \end{equation}

\end{lemma}
\begin{proof}
First one may use Markov inequality and Lemma \ref{lemma: strong cont moment bounds} to obtain for any $q, \epsilon>0$ that 
    \begin{equation}\label{eq: decay of stopping times}
  P(\inf_{t\in [0,T]}\rho(X_{t\wedge \tau_{\epsilon}})\leq \epsilon)=P(\sup_{t\in [0,T]}\rho(X_{t\wedge \tau_{\epsilon}})^{-q}\geq \epsilon^{-q})\leq c\epsilon^q.
    \end{equation}
Now let us define
\begin{equation}\label{eq: defn of cont process}
    X_t=\lim_{\epsilon\to 0}X_{t\wedge \tau_{\epsilon}}.
\end{equation}
This defines a strong solution to \eqref{eq: SDE main} on $[0,T]$ providing $P(\tau_\epsilon\in [0,T]\;\text{for every}\; \epsilon>0)=0$. To this end one has
\begin{equation}
    P(\tau_\epsilon\in [0,T]\;\text{for every}\; \epsilon>0)\leq \lim_{\epsilon\to 0} P(\tau_\epsilon\in [0,T])=\lim_{\epsilon\to 0}P(\inf_{t\in [0,T]}\rho(X_{t\wedge \tau_{\epsilon}})\leq \epsilon),
\end{equation}
so by \eqref{eq: decay of stopping times}, taking the limit as $\epsilon\to 0$ the strong solution to \eqref{eq: SDE main} on $[0,T]$ follows. To see that this extends to a strong solution on the whole of $[0,\infty)$, note that since we can repeat this construction for any $T>0$, the largest interval $[0,T]$ for which \eqref{eq: SDE main} has a strong solution cannot be finite. 
    \end{proof}
    \begin{lemma}\label{lemma: inverse moment bound cont}
        For every $q>0$ and $T>0$ there exists $c>0$ such that for every $\epsilon>0$ one has
    \begin{equation}\label{eq: epsilon decay cont sol}
        P(\inf_{t\in [0,T]}\rho(X_t)\leq \epsilon)\leq c\epsilon^q.
    \end{equation}
Furthermore, for every $p>0$ one has
    \begin{equation}\label{eq: inverse moments bounds for rho}
       E\sup_{t\in [0,T]}\rho(X_t)^{-p}<\infty.
       \end{equation}
       \end{lemma}
       \begin{proof}
           By Lemma \ref{lemma: true solution cont for all time} we know that there exists a strong solution of \eqref{eq: SDE main} on the whole of $[0,T]$, so one may observe that by continuity
    \begin{equation}
        P(\inf_{t\in [0,T]}\rho(X_t)\leq \epsilon)= P(\inf_{t\in [0,T]}\rho(X_{t\wedge \tau_{\epsilon}})\leq \epsilon),
    \end{equation}
    so \eqref{eq: epsilon decay cont sol} follows from \eqref{eq: decay of stopping times}. Furthermore, one sees that the second result follows from \eqref{eq: defn of cont process}, Lemma \ref{lemma: strong cont moment bounds} and Fatou's Lemma.
       \end{proof}
    \begin{lemma}\label{lemma: cont regular moments}
  For every $p>0$ and $T>0$ one has
           \begin{equation}\label{eq: moments bounds for cont sol}
       E\sup_{t\in [0,T]}\lvert X_t\rvert^p< \infty.
    \end{equation}
    \end{lemma}
    \begin{proof}
    Let us define
    \begin{equation}
        \tilde{\tau}_R:=\inf\{t\geq 0 \vert \; \lvert X_t\rvert \geq R \}
    \end{equation}
By Itô's formula there exists a constant $c>0$ independent of $R>0$ and $t>0$ such that
    \begin{align}
  \lvert X_{t\wedge \tilde{\tau}_R} \rvert^p &\leq \lvert x_0\rvert^p+\int^{t\wedge \tilde{\tau}_R}_0 (p \lvert X_s\rvert^{p-2}\langle X_s, b(X_s)\rangle+c\lvert X_s\rvert^{p-2})ds\nonumber \\
  &+\int^{t\wedge \tilde{\tau}_R}_0p \lvert X_s\rvert^{p-2}\langle X_s, dW_s\rangle.
    \end{align}
    Now observe that, fixing arbitrary $y_0\in D$, one has by \eqref{assmp: monotonicity} and Young's inequality that there exists a constant $c>0$ such that for every $x\in D$
     \begin{align}\label{eq: b bound using mon}
     \langle x, b(x)\rangle & = \langle x-y_0, b(x)-b(y_0)\rangle+\langle y_0, b(x)-b(y_0)\rangle+\langle x, b(y_0)\rangle \nonumber \\
     &\leq c(1+\lvert  b(x)\rvert +\lvert x \rvert^2).
    \end{align}
Applying this bound and squaring, one obtains for every $t\in [0,T]$ that
    \begin{align}
  \lvert X_{t\wedge \tilde{\tau}_R} \rvert^{2p}\leq c\lvert &x_0\rvert^{2p}+c\int^{t\wedge \tilde{\tau}_R}_0 \lvert X_s\rvert^{2p-4}(1+\lvert  b(X_s)\rvert^2 +\lvert X_s \rvert^4)ds\nonumber \\
  &+c\biggr (\int^{t\wedge \tilde{\tau}_R}_0p \lvert X_s\rvert^{p-2}\langle X_s, dW_s\rangle\biggr)^2.
    \end{align}
     Furthermore, using \eqref{eq: upper bound for b} there exists another constant $c>0$ such that for all $x\in D$
    \begin{equation}\label{eq: b upper bound 2}
        \lvert  b(x)\rvert\leq c(1+ \rho(x)^{-2l}+\lvert x\rvert^2),
    \end{equation}
    and therefore
  \begin{align}
  \lvert X_{t\wedge \tilde{\tau}_R} \rvert^{2p}&\leq c\lvert x_0\rvert^{2p}+c\int^{t\wedge \tilde{\tau}_R}_0 \biggr(1+\lvert X_s\rvert^{2p} +\lvert X_s\rvert^{2p-4}\rho(X_s)^{-4l}\biggr)ds\nonumber \\
  &+c\biggr (\int^{t\wedge \tilde{\tau}_R}_0p \lvert X_s\rvert^{p-2}\langle X_s, dW_s\rangle\biggr)^2\nonumber \\
  &\leq c\lvert x_0\rvert^{2p}+c\int^{t\wedge \tilde{\tau}_R}_0 (1+\sup_{u\in [0,s]}\lvert X_{u\wedge \tilde{\tau}_R}\rvert^{2p} +\sup_{u\in [0,s]}\rho(X_{u\wedge \tilde{\tau}_R})^{-2pl} )ds\nonumber \\
  &+c\sup_{u\in [0,t]}\biggr (\int^{u\wedge \tilde{\tau}_R}_0p \lvert X_s\rvert^{p-2}\langle X_s, dW_s\rangle\biggr)^2.
    \end{align}
    Therefore, taking supremum over the LHS, applying Lemma \ref{lemma: inverse moment bound cont}, the Burkholder-Davis-Gundy Inequality and Young's inequality, one obtains
      \begin{align}
 E\sup_{u\in [0,t]} \lvert X_{u\wedge \tilde{\tau}_R} \rvert^{2p}\leq c\lvert x_0\rvert^{2p}+c\int^t_0 (1+E\sup_{u\in [0,s]}\lvert X_{u\wedge \tilde{\tau}_R}\rvert^{2p} )ds.
    \end{align}
    Now observe that since 
    \begin{equation}
    E\sup_{u\in [0,t]} \lvert X_{u\wedge \tilde{\tau}_R} \rvert^{2p}<\infty,
    \end{equation}
    by the definition of $\tilde{\tau}_R$, one may use Grönwall's inequality to conclude that there exists a constant independent of $R,T>0$ such that
          \begin{align}
 E\sup_{u\in [0,T]} \lvert X_{u\wedge \tilde{\tau}_R} \rvert^{2p}\leq c.
    \end{align}
    Then by continuity 
    \begin{equation}
    \lim_{R\to\infty}\sup_{u\in [0,t]} \lvert X_{u\wedge \tilde{\tau}_R}\rvert=\sup_{u\in [0,t]} \lvert X_u \rvert,
    \end{equation}
    almost surely, so that the result follows from Fatou's lemma.
\end{proof}
\section{Tamed Euler Scheme}
In this section we aim to replicate the result of the previous section for the tamed Euler scheme introduced in Theorem \ref{thm: main theorem}. Owing to the discretisation structure of the scheme, we shall begin with the following apriori moment bound for $X^n_t$, which we upgrade in Lemma \ref{lemma: regular moment bounds scheme good} to uniform moment bounds up to a certain stopping time. In particular we define
\begin{equation}\label{eq: tau n defn}
    \tau^n:=\inf\{ t\geq 0 \vert \; \rho(X^n_t)\leq \delta n^{-w}\},
\end{equation}
so that it immediately follows that
\begin{equation}\label{eq: rho bound}
\rho(X^n_{t\wedge \tau^n})\geq \rho( x_0) \wedge \delta n^{-w},
\end{equation}
for all $t\geq 0$. Furthermore, by the definition of $b_n$ as given in Theorem \ref{thm: main theorem}, one has for all $t\geq 0$ that
\begin{equation}\label{eq: b_n equal to b before stopping}
    b_n(X^n_{t\wedge \tau^n})1_{t\in [0,\tau^n)}=b(X^n_{t\wedge \tau^n})1_{t\in [0,\tau^n)}.
\end{equation}
As before we assume that the assumptions given in Section \ref{sec: assump and main thm} hold. 
\begin{lemma}\label{lemma: bad regular moment bounds scheme}
For every $p>0$ and $T>0$ there exists a $c>0$ such that for every $n\geq 1$
\begin{equation}
  E \sup_{s\in [0,T]}\lvert X^n_s \rvert^{p}\leq c(1+n^{lwp}).
\end{equation}
\end{lemma}
\begin{proof}
    Now let us follow the strategy of the proof of Lemma \ref{lemma: cont regular moments}, and define the stopping time $\tilde{\tau}_R^n$ by
        \begin{equation}
     \tilde{\tau}_R^n:=\inf\{t\geq 0\vert \; \lvert X^n_t\rvert\geq R\}. 
    \end{equation} 
    Then for $p\geq 2$ one calculates using Itô's formula and \eqref{eq: b_n equal to b before stopping} that there exists $c>0$ such that
     \begin{align}\label{eq: Ito for stopped}
  \lvert X^n_{t\wedge \tilde{\tau}_R^n} \rvert^{p}\leq \lvert x_0\rvert^p&+\int^{t \wedge \tilde{\tau}_R^n}_0 (p\lvert X^n_s\rvert^{p-2}\langle X^n_s, b_n(X^n_{\kappa_n(s)}\rangle +c\lvert X^n_s\rvert^{p-2})ds \nonumber \\
  &+\int^{t \wedge \tilde{\tau}_R^n}_0 p \lvert X^n_s\rvert^{p-2}\langle X_s, dW_s\rangle ds.
    \end{align}
    For the first term we bound as
    \begin{align}\label{eq: splitting for proof}
        \lvert X^n_s\rvert^{p-2} \langle X^n_s, b_n(X^n_{\kappa_n(s)}) \rangle &\leq    \lvert X^n_s\rvert^{p-2}\langle X^n_{\kappa_n(s)}, b_n(X^n_{\kappa_n(s)})\rangle\nonumber \\
        &+ \lvert X^n_s\rvert^{p-2}\langle X^n_s-X^n_{\kappa_n(s)}, b_n(X^n_{\kappa_n(s)})\rangle
    \end{align}
    To bound the first term in the splitting, one applies the definition of $b_n$ as well as \eqref{eq: b bound using mon} and Young's inequality to obtain
   \begin{align}\label{eq: splitting 1}
     &\lvert X^n_s\rvert^{p-2}\langle X^n_{\kappa_n(s)}, b_n(X^n_{\kappa_n(s)})\rangle \nonumber \\
     & =  \lvert X^n_s\rvert^{p-2}\langle X^n_{\kappa_n(s)}, b(X^n_{\kappa_n(s)})\rangle 1_{\{ X^n_{\kappa_n(s)}\in D\} \cap\{\rho(X^n_{\kappa_n(s)})\geq \delta n^{-w}\}} \nonumber \\
     &\leq c(1+  \lvert X^n_s\rvert^p+ \lvert X^n_{\kappa_n(s)}\rvert^p)\nonumber \\
     &+  \lvert b(X^n_{\kappa_n(s)})\rvert^{p/2}1_{\{ X^n_{\kappa_n(s)}\in D\} \cap\{\rho(X^n_{\kappa_n(s)})\geq \delta n^{-w}\}}  \nonumber \\
     &\leq c(1+  \lvert X^n_s\rvert^p+ \lvert X^n_{\kappa_n(s)}\rvert^p+\lvert b_n(X^n_{\kappa_n(s)})\rvert^{p/2}).
   \end{align}
   Similarly, for the second term on the RHS of \eqref{eq: splitting for proof}, one may bound the increment to obtain
   \begin{align}
       \lvert X^n_s\rvert^{p-2}\langle X^n_s&-X^n_{\kappa_n(s)}, b_n(X^n_{\kappa_n(s)})\rangle\nonumber \\
       &\leq c\lvert X^n_s\rvert^{p-2}n^{-1}\lvert b(X^n_{\kappa_n(s)}) \rvert^2 1_{\{ X^n_{\kappa_n(s)}\in D\} \cap\{\rho(X^n_{\kappa_n(s)})\geq \delta n^{-w}\}} \nonumber \\
       &+c\lvert X^n_s\rvert^{p-2}\lvert b(X^n_{\kappa_n(s)}) \rvert \lvert  W_s-W_{\kappa_n(s)}\rvert 1_{\{ X^n_{\kappa_n(s)}\in D\} \cap\{\rho(X^n_{\kappa_n(s)})\geq \delta n^{-w}\}}.
   \end{align}
Therefore, applying \eqref{eq: upper bound for b} and \eqref{eq: rho bound} one has
\begin{align}
       &\lvert X^n_s\rvert^{p-2}\langle X^n_s-X^n_{\kappa_n(s)}, b_n(X^n_{\kappa_n(s)})\rangle\nonumber \\
       &\leq cn^{-1}\lvert X^n_s\rvert^{p-2}(1+n^{2lw})(1+\lvert X^n_{\kappa_n(s)}\rvert^2)\nonumber \\
       &+c\lvert X^n_s\rvert^{p-2}(1+n^{lw})(1+\lvert X^n_{\kappa_n(s)}\rvert)\lvert  W_s-W_{\kappa_n(s)}\rvert,
   \end{align}
   and since  $lw\leq 1/2$ by assumption, one may apply \eqref{eq: rho bound} and Young's inequality to obtain
\begin{align}\label{eq: splitting 2}
       &\lvert X^n_{s}\rvert^{p-2}\langle X^n_{s}-X^n_{\kappa_n(s)}, b(X^n_{\kappa_n(s)})\rangle\nonumber \\
       &\leq c(1+\lvert X^n_{s}\rvert^p+\lvert X^n_{\kappa_n(s)}\rvert^p+n^{p/2}\lvert  W_s-W_{\kappa_n(s)}\rvert^p).
   \end{align}
   Therefore, one may substitute in \eqref{eq: splitting for proof}, \eqref{eq: splitting 1} and \eqref{eq: splitting 2}
       \begin{align}
  \lvert X^n_{t\wedge \tilde{\tau}_R^n} \rvert^{p}&\leq \lvert x_0\rvert^p+\int^{t\wedge \tilde{\tau}_R^n}_0 c(1+\lvert X^n_{s}\rvert^p+\lvert X^n_{\kappa_n(s)}\rvert^p)ds\nonumber \\
  &+\int^{t\wedge \tilde{\tau}_R^n}_0c(\lvert b_n(X^n_{\kappa_n(s)})\rvert^{p/2}+n^{p/2}\lvert  W_s-W_{\kappa_n(s)}\rvert^p)ds \nonumber \\
  &+\int^{t\wedge \tilde{\tau}_R^n}_0 p \lvert X^n_s\rvert^{p-2}\langle X_s, dW_s\rangle
    \end{align}
     and therefore taking supremum on the RHS then the LHS, squaring and applying H{\"o}lder's inequality, one may write
        \begin{align}
  \sup_{u\in [0,t]}\lvert X^n_{u\wedge \tilde{\tau}_R^n} \rvert^{2p}&\leq c\lvert x_0\rvert^{2p}+\int^{t\wedge \tilde{\tau}_R^n}_0 c(1+\sup_{u\in [0,s]}\lvert X^n_u\rvert^{2p})ds\nonumber \\
  &+\int^{t\wedge \tilde{\tau}_R^n}_0c(\lvert b_n(X^n_{\kappa_n(s)})\rvert^p+n^{p}\lvert  W_s-W_{\kappa_n(s)}\rvert^{2p})ds \nonumber \\
  &+c\sup_{u\in [0,t]}\biggr (\int^{u\wedge \tilde{\tau}_R^n}_0  \lvert X^n_s\rvert^{p-2}\langle X_s, dW_s\rangle\biggr)^2 \nonumber \\
  &= c\lvert x_0\rvert^{2p}+\int^{t\wedge \tilde{\tau}_R^n}_0 c(1+\sup_{u\in [0,s]}\lvert X^n_{u\wedge \tilde{\tau}_R^n}\rvert^{2p})ds\nonumber \\
  &+\int^{t\wedge \tilde{\tau}_R^n}_0c(\lvert b_n(X^n_{\kappa_n(s)\wedge \tilde{\tau}_R^n})\rvert^p+n^{p}\lvert  W_s-W_{\kappa_n(s)}\rvert^{2p})ds \nonumber \\
  &+c\sup_{u\in [0,t]}\biggr (\int^{u\wedge \tilde{\tau}_R^n}_0  \lvert X^n_{s\wedge \tilde{\tau}_R^n}\rvert^{p-2}\langle X_{s\wedge \tilde{\tau}_R^n}, dW_s\rangle\biggr)^2,
    \end{align}
    and therefore by the Burkholder-Davis-Gundy inequality one has by standard properties of a Wiener martingale that
     \begin{align}\label{eq: first moment bound scheme}
  E\sup_{u\in [0,t]}&\lvert X^n_{u\wedge \tilde{\tau}_R^n} \rvert^{2p}\leq \lvert x_0\rvert^p+\int^t_0 c(1+E\sup_{u\in [0,s]}\lvert X^n_{u\wedge \tilde{\tau}_R^n}\rvert^{2p})ds\nonumber \\
  &+\int^t_0cE\lvert b_n(X^n_{\kappa_n(s)\wedge \tilde{\tau}_R^n})\rvert^pds 
    \end{align}
       Now note by \eqref{eq: b upper bound 2}, \eqref{eq: rho bound} and the definition of $b_n$ in Theorem \ref{thm: main theorem}, one has
 \begin{equation}
        \lvert  b_n (X^n_{\kappa_n(s)\wedge \tilde{\tau}_R^n})\rvert\leq c(1+n^{2plw}+\lvert X^n_{\kappa_n(s)\wedge \tilde{\tau}_R^n} \rvert^2),
    \end{equation}
    and so
    \begin{align}
  E\sup_{u\in [0,t]}&\lvert X^n_{u\wedge \tilde{\tau}_R^n} \rvert^{2p}\leq \lvert x_0\rvert^p+\int^t_0 c(1+n^{2wl}+E\sup_{u\in [0,s]}\lvert X^n_{u\wedge \tilde{\tau}_R^n}\rvert^{2p})ds.
    \end{align}  
Therefore observing that 
\begin{equation}
\sup_{t\in [0,T]} \lvert X^n_{t\wedge \tilde{\tau}_R^n} \rvert^{p}\leq R\wedge \lvert x_0\rvert<\infty,
\end{equation}
almost surely by definition, the result follows by Grönwall's inequality, and then taking $R\to \infty$ (since $\lim_{R\to\infty} X^n_{t\wedge \tilde{\tau}_R^n} \to X^n_t$) and applying Fatou's lemma.
\end{proof}
Now we prove moments inverse moments of $\rho(X^n_t)$ up to $\tau^n$ given in \eqref{eq: tau n defn}.
\begin{lemma}\label{lemma: strong scheme inverse moment bounds}
    For every $p>0$, $T>0$ and there exists a constant $c>0$ such that
       \begin{align}\label{eq: strong scheme bound}
    \sup_{n \geq 1} E\sup_{t\in [0,T]}\rho(X^n_{t\wedge \tau^n})^{-p}<\infty.
    \end{align}
\end{lemma}
\begin{proof}
Applying Proposition \ref{prop: local Ito} for $\epsilon=\delta n^{-w}$, by \eqref{eq: b_n equal to b before stopping} one sees that
         \begin{align}\label{eq: scheme negative moments}
    \rho(X^n_{t\wedge \tau^n})^{-p}&=\rho(x_0)^{-p}-p\int^{t\wedge \tau^n}_0\rho(X^n_s)^{-p-1}\langle \nabla \rho(X^n_s), b(X^n_{\kappa_n(s)}) \rangle ds\nonumber \\
    &-\frac{p}{2} \int^{t\wedge \tau^n}_0 \rho(X^n_s)^{-p-1}tr(\nabla ^2\rho(X^n_s))ds \nonumber \\
    &+\frac{p(p+1)}{2}\int^{t\wedge \tau^n}_0\rho(X^n_s)^{-p-2}tr(\nabla \rho(X^n_s)\otimes \nabla \rho(X^n_s))   ds\nonumber \\
    &-p\int^{t\wedge \tau^n}_0\rho(X^n_s)^{-p-1}\langle \nabla \rho(X^n_s),dW_s \rangle .
    \end{align}
    Let us first control the first integral. By \eqref{assmp: rho pushes away} and Young's inequality, and since by \eqref{eq: Lipschitz calc} one has $\lvert \nabla \rho\rvert\leq 1$, for $\beta>0$ as given in Section \ref{sec: assump and main thm} one may write 
    \begin{align}
&-p\int^{t\wedge \tau^n}_0\rho(X^n_s)^{-p-1}\langle \nabla \rho(X^n_s), b(X^n_{\kappa_n(s)}) \rangle ds\nonumber \\
&\leq -p\int^{t\wedge \tau^n}_0\rho(X^n_s)^{-p-1}\langle \nabla \rho(X^n_s), b(X^n_s) \rangle ds \nonumber \\
&+p\int^{t\wedge \tau^n}_0\rho(X^n_s)^{-p-1}\langle \nabla \rho(X^n_s), b(X^n_s) -b(X^n_{\kappa_n(s)})  \rangle ds \nonumber 
\\
&\leq  -p\int^{t\wedge \tau^n}_0\rho(X^n_s)^{-p-1}\langle \nabla \rho(X^n_s), b(X^n_s) \rangle ds \nonumber \\
&+c \int^{t\wedge \tau^n}_0 (\lvert b(X^n_{\kappa_n(s)}) - b(X^n_s) \rvert^{\frac{p+1}{\beta+1}+1} +\rho(X_s)^{-p-2-\beta})ds.
    \end{align}
Therefore, as in \eqref{eq: Ito cont with bound applied}, using \eqref{assmp: rho pushes away} and \eqref{assmp: second derivative rho} one may find $c_1, c_2, c_3, c_4>0$ not dependent on $n\in \mathbb{N}$ or $t\in [0,T]$ such that
           \begin{align}\label{eq: Ito scheme with bound applied}
    &\rho(X^n_{t\wedge \tau^n})^{-p}\leq\rho(x_0)^{-p}+\int^{t\wedge \tau^n}_0(-c_1\rho(X^n_s)^{-p-1-\alpha}+c_2\rho(X^n_s)^{-p-1-\beta})ds\nonumber \\
   & +\int^{t\wedge \tau^n}_0(c_3\rho(X^n_s)^{-p-1}+c_4\rho(X^n_s)^{-p-1-\beta})ds \nonumber \\
    &+\int^{t\wedge \tau^n}_0c_5 \rho(X^n_s)^{-p}ds +\int^{t\wedge \tau^n}_0 c\lvert b(X^n_{\kappa_n(s)}) - b(X^n_s)\rvert^{\frac{p+1}{\beta+1}+1} ds\nonumber \\
    &-p\int^{t\wedge \tau^n}_0\rho(X_s)^{-p-1}\langle \nabla \rho(X_s),dW_s \rangle .
    \end{align}
As a result, using \eqref{eq: bound on polynomial}, there exists a constant $c>0$ independent of $t>0$ and $n\geq 1$ such that
        \begin{align}\label{eq: moment rho bounds scheme calc}
    \rho(X^n_{t\wedge \tau^n})^{-p}&\leq\rho(x_0)^{-p}+ct+\int^{t\wedge \tau^n}_0c_5 \rho(X^n_{s\wedge \tau^n})^{-p}ds\nonumber \\
    &+\int^{t}_0 c \lvert b(X^n_{\kappa_n(s)\wedge \tau^n}) - b(X^n_{s\wedge \tau^n})\rvert ^{\frac{p+1}{\beta+1}+1} ds\nonumber \\
    &-p\int^{t\wedge \tau^n}_0\rho(X_s)^{-p-1}\langle \nabla \rho(X_s),dW_s \rangle .
    \end{align}
Furthermore, applying \eqref{assmp: inverse poly Lipschitz b} and \eqref{eq: rho bound} one has
    \begin{align}\label{eq: final calc scheme inverse moments}
        & \lvert b(X^n_{\kappa_n(s)\wedge \tau^n}) - b(X^n_{s\wedge \tau^n})\rvert\nonumber \\
        &\leq c(1+n^{lw})(n^{-1}\lvert b(X^n_{\kappa_n(s)\wedge \tau^n}) \rvert+\lvert W_{s\wedge \tau^n}-W_{\kappa_n(s)\wedge \tau^n}\rvert).
    \end{align}
    Additionally, by \eqref{eq: upper bound for b} and \eqref{eq: rho bound}
     \begin{equation}
         \lvert b(X^n_{\kappa_n(s)\wedge \tau^n}) \rvert\leq c(1+n^{lw})(1+\lvert X^n_{\kappa_n(s)\wedge \tau^n}\rvert),
    \end{equation}
   and so
    \begin{align}\label{eq: incremement bdd for stopped scheme}
    \lvert b(X^n_{\kappa_n(s)\wedge \tau^n})& - b(X^n_{s\wedge \tau^n})\rvert \leq c(1+n^{2lw-1})(1+\lvert X^n_{\kappa_n(s)\wedge \tau^n}\rvert)\nonumber \\
    &+c(1+n^{lw})\lvert W_{s\wedge \tau^n}-W_{\kappa_n(s)\wedge \tau^n}\rvert.
    \end{align}
    Now to bound the increment of the stochastic integral note that
    \begin{align}\label{eq: stopped increment of Wiener martingale}
        \lvert W_{s\wedge \tau^n}-W_{\kappa_n(s)\wedge \tau^n}\rvert \leq \sup_{u\in [\kappa_n(s), \kappa_n(s)+n^{-1}]}\lvert W_u-W_{\kappa_n(s)}\rvert,
    \end{align}
  and therefore, for any $q>0$, by Lemma \ref{lemma: bad regular moment bounds scheme} one has
 \begin{align}
    E\lvert b(X^n_{\kappa_n(s)\wedge \tau^n}) - b(X^n_{s \wedge \tau^n})\rvert^q\leq c(n^{q(3lw-1)}+1).
    \end{align}
Then since by assumption $lw\leq 1/3$, one obtains that the RHS is bounded independent of $n$. Noting that the stochastic integral in \eqref{eq: moment rho bounds scheme calc} must be a martingale since the integrand is bounded by the definition of $\tau^n$, we may apply expectation to \eqref{eq: moment rho bounds scheme calc}, and then Grönwall's inequality, to obtain
        \begin{align}
    \sup_{n\geq 1}\sup_{t\in [0,T]}E\rho(X^n_{t\wedge \tau^n})^{-p}<\infty.
    \end{align}
Therefore, proceeding similarly to the conclusion of the proofs of Lemma \ref{lemma: strong cont moment bounds} and Lemma \ref{lemma: cont regular moments}, we may square \eqref{eq: moment rho bounds scheme calc}, apply supremum on both sides, and then apply expectation, the Burkholder-David-Gundy inequality and Grönwall's inequality, to obtain the result.
\end{proof}
We may now upgrade the bound in Lemma \ref{lemma: bad regular moment bounds scheme} to a bound which is uniform in $n\geq 1$, up to $\tau^n$ given in \eqref{eq: tau n defn}.
\begin{lemma}\label{lemma: regular moment bounds scheme good}
    For every $p>0$ and $T>0$ there exists a constant $c>0$ such that
       \begin{align}\label{eq: regular moment bounds scheme good}
    \sup_{n \geq 1} E\sup_{t\in [0,T]}\lvert X^n_{t\wedge \tau^n}\rvert^p <\infty.
    \end{align}
\end{lemma}
\begin{proof}
Using \eqref{eq: b upper bound 2} and \eqref{eq: b_n equal to b before stopping}, by the definition of $b_n$ one has
\begin{equation}
    \lvert b_n(X^n_{\kappa_n(s)\wedge \tau^n})\rvert \leq   \lvert b(X^n_{\kappa_n(s)\wedge \tau^n})\rvert\leq c(1+\rho(X^n_{\kappa_n(s)\wedge \tau^n})^{-2l}+\lvert X^n_{\kappa_n(s)\wedge \tau^n}\rvert^2).
\end{equation}
so that by Lemma \ref{lemma: strong scheme inverse moment bounds} one sees that
\begin{equation}\label{eq: bn expression}
    E\lvert b_n(X^n_{\kappa_n(s)\wedge \tau^n})^p\rvert \leq c(1+E\lvert X^n_{\kappa_n(s)\wedge \tau^n}\rvert^{2p}).
\end{equation}
Now observe that \eqref{eq: first moment bound scheme} holds with $\tilde{\tau}^n_R$ replaced with $\tau^n$, since the calculations prior hold almost surely. Therefore substituting in \eqref{eq: bn expression} and applying Grönwall's inequality the result follows.
\end{proof}
\section{Proof of Theorem \ref{thm: main theorem}}
The fact $X_t$ is well defined for all $t\geq 0$ follows from Lemma \ref{lemma: true solution cont for all time}, and the fact $X^n_t$ is well defined is obvious by the Euler scheme structure. To prove convergence, let us assume $p\geq 2$, so that one may apply the chain rule and \eqref{eq: b_n equal to b before stopping} to obtain that
    \begin{align}
        \lvert X_{t\wedge \tau^n}-X^n_{t\wedge \tau^n}\rvert^p&=\int^{t\wedge \tau^n}_0 p \lvert X_s-X^n_s\rvert^{p-2}\langle X_s-X^n_s, b(X_s)-b(X^n_{\kappa_n(s)})\rangle ds \nonumber \\
        &=\int^{t\wedge \tau^n}_0 p \lvert X_s-X^n_s\rvert^{p-2}\langle X_s-X^n_s, b(X_s)-b(X^n_s)\rangle ds\nonumber \\
        &+\int^{t\wedge \tau^n}_0 p \lvert X_s-X^n_s\rvert^{p-2}\langle X_s-X^n_s, b(X^n_s)-b(X^n_{\kappa_n(s)}) \rangle ds .\nonumber 
    \end{align}
    Then applying \eqref{assmp: monotonicity} and Young's inequality, one sees that there exists $c>0$ independent of $t\geq 0$ and $n\geq 1$ such that
    \begin{align}
        \lvert X_{t\wedge \tau^n}&-X^n_{t\wedge \tau^n}\rvert^p=\int^{t\wedge \tau^n}_0 c (\lvert X_s-X^n_s\rvert^p+\lvert b(X^n_{\kappa_n(s)})-b(X^n_s) \rvert^p) ds  \nonumber \\
        &\leq \int^{t\wedge \tau^n}_0 c (\sup_{u\in [0,s]}\lvert X^n_{u \wedge \tau^n}-X^n_{u \wedge \tau^n}\rvert^p+\lvert b(X^n_{\kappa_n(s)\wedge \tau^n})-b(X^n_{s\wedge \tau^n}) \rvert^p) ds.
    \end{align} 
    Therefore, since the RHS is increasing in $t\geq0$ one has
   \begin{align}
       \sup_{u\in [0,t]} &\lvert X_{u\wedge \tau^n}-X^n_{u\wedge \tau^n}\rvert^p \nonumber \\
       &\leq \int^t_0 c (\sup_{u\in [0,s]}\lvert X_{u\wedge \tau^n}-X^n_{u\wedge \tau^n}\rvert^p+\lvert b(X^n_{\kappa_n(s)\wedge \tau^n})-b(X^n_{s\wedge \tau^n}) \rvert^p) ds.
    \end{align}    
   Now let us calculate by \eqref{assmp: inverse poly Lipschitz b}, \eqref{eq: b upper bound 2}, \eqref{eq: stopped increment of Wiener martingale}, Lemma \ref{lemma: strong scheme inverse moment bounds} and Lemma \ref{lemma: regular moment bounds scheme good} that
 \begin{align}
    &E\lvert b(X^n_{\kappa_n(s)\wedge \tau^n})-b(X^n_{s\wedge \tau^n}) \rvert^p \nonumber \\
    &\leq \biggr [E\biggr (1+ \rho(X^n_{s\wedge \tau^n})^{-l}+ \rho(X^n_{\kappa_n(s)\wedge \tau^n})^{-l}\biggr)^{2p} \biggr ]^{1/2}\nonumber \\
    &\times \biggr [ E\biggr (n^{-1}\lvert b(X^n_{\kappa_n(s)\wedge \tau^n}) \rvert+\lvert W_{s\wedge \tau^n}-W_{\kappa_n(s)\wedge \tau^n}\rvert \biggr)^{2p} \biggr ]^{1/2}\nonumber \\
    &\leq cn^{-p/2},
    \end{align} 
    and therefore
    \begin{align}
       E\sup_{u\in [0,t\wedge \tau^n]}& \lvert X_{u\wedge \tau^n}-X^n_{u\wedge \tau^n}\rvert^p\nonumber \\
       &\leq \int^{t\wedge \tau^n}_0 c (E\sup_{u\in [0,s]}\lvert X_{u\wedge \tau^n}-X^n_{u\wedge \tau^n} \rvert^p+n^{-p/2} ) ds.
    \end{align} 
    Now finally note that by Lemma \ref{lemma: cont regular moments} and Lemma \ref{lemma: strong scheme inverse moment bounds}, one has that 
    \begin{equation}
    E\sup_{u\in [0,T]} \lvert X_{u\wedge \tau^n}-X^n_{u\wedge \tau^n}\rvert^p<\infty, 
    \end{equation}
    so by Grönwall's inequality there exists $c>0$ independent of $n\geq 1$ such that
    \begin{align}\label{eq: main theorem 1}
       E\sup_{u\in [0,T]} \lvert X_{u\wedge \tau^n}-X^n_{u\wedge \tau^n}\rvert^p\leq cn^{-p/2}.
    \end{align} 
   Now note that for any $q>0$, by Markov's inequality and Lemma \ref{lemma: strong scheme inverse moment bounds} there exists $c>0$ independent of $n\geq 1$ such that
\begin{equation}\label{eq: proof stopping time unlikely}
    P(\tau^n\leq T)\leq P(\sup_{t\in [0,T]} \rho(X^n_{t\wedge \tau^n})^{-q}>\delta^{-q} n^{wq})\leq cn^{-wq}.
\end{equation}
Therefore since by Lemmas \ref{lemma: bad regular moment bounds scheme} and \ref{lemma: regular moment bounds scheme good}
\begin{align}
    E\sup_{u\in [0,T]}\lvert X^n_{u\wedge \tau^n}-X^n_u \rvert^p &= E\sup_{u\in [0,T]}\lvert X^n_{u\wedge \tau^n}-X^n_u \rvert ^p 1_{\tau^n\leq T} \nonumber \\
    &\leq [E(\sup_{u\in [0,T]}\lvert X^n_{u\wedge \tau^n}\rvert^{2p}+\sup_{u\in [0,T]}\lvert X^n_u \rvert^{2p})P(\tau^n\leq T)]^{1/2}\nonumber \\
    &\leq c(1+n^{pwl})P(\tau^n\leq T)^{1/2},
\end{align}
by setting $q>0$ sufficiently large in \eqref{eq: proof stopping time unlikely} it follows that
\begin{equation}\label{eq: main theorem 2}
     E\sup_{u\in [0,T]}\lvert X^n_{u\wedge \tau^n}-X^n_u \rvert^p\leq cn^{-p/2}.
\end{equation}
 Similarly one may use Lemma \ref{lemma: cont regular moments} to conclude
 \begin{equation}\label{eq: main theorem 3}
     E\sup_{u\in [0,T]}\lvert X_{u\wedge \tau^n}-X_u \rvert^p\leq cn^{-p/2}.
\end{equation}
The result then follows from the triangle inequality by combining \eqref{eq: main theorem 1}, \eqref{eq: main theorem 2} and \eqref{eq: main theorem 3}.

\section{Interacting Particles with Singular Interaction Kernal}\label{sec: Interacting Particles}
In this section we show that the hypothesis of Theorem \ref{thm: main theorem} is satisfied in the case of two scalar valued interacting particles with singular interactions. Interacting particles with singular drift is a key example and a primary motivation for the study of SDEs with non-locally integrable drift. In future work we hope to consider more general settings, in particular where \eqref{assmp: monotonicity} and \eqref{assmp: rho pushes away} are replaced by more general assumptions.

For Corollary \ref{cor: interacting particles} one may take $K$ equal to the gradient of the well-known potential known as the Lennard-Jones interaction, used widely in molecular dynamics. Specifically, for $a_1,a_2>0$ and $p>1$, $0<q<p$ one sets $K:(0,\infty)\to \mathbb{R}$ to be
\begin{equation}
    K(x)=-\nabla (a_1 x^{-p}-a_2x^{-q})(x)=(pa_1 x^{-p-1}-qa_2x^{-q-1}).
\end{equation}
See \cite{29acd3d494044594aea0829ef236aad6} for more information. The function $Q$ in Corollary \ref{cor: interacting particles} is usually given as the gradient of a `confining potential', and ensures the particles do not escape to infinity. One could choose for instance $Q(x)=-\lambda x$ for some $\lambda>0$. Our assumptions also cover the family of interacting particles systems considered in \cite{Guillin2022OnSO} for $\alpha>2$ and $N=2$.

\begin{corollary}\label{cor: interacting particles}
    Consider interacting particles $X^{i,N}$, $i=1,2$ taking values in $\mathbb{R}$, given as 
\begin{equation}\label{eq: particle interactions 1}
    dX^{1,N}_t=Q(X^{1,N})-K(X^{2,N}-X^{1,N})dt+dW^1_t,\;\;\; X^{1,N}_0=x^1_0.
\end{equation}
\begin{equation}\label{eq: particle interactions 2}
    dX^{2,N}_t=Q(X^{2,N})+K(X^{2,N}-X^{1,N})dt+dW^2_t,\;\;\; X^{2,N}_0=x^2_0.
\end{equation}
for $x^1_0< x^2_0$, a Lipschitz function $Q:\mathbb{R}\to\mathbb{R}$, and some interaction kernel $K:(0,\infty)\to\mathbb{R}$ for which we assume that there exists $c,l>0$ such that for all $x,y>0$ 
\begin{equation}\label{eq: K monotonicity}
    (K(x)-K(y))(x-y)\leq c\lvert x -y\rvert^2,
\end{equation}
\begin{equation}\label{eq: K inverse poly Lipschitz}
    \lvert K(x)-K(y)\rvert\leq c(1+x^{-l}+y^{-l})\lvert x-y \rvert.
\end{equation}
Assume additionally that there exist $h_1, h_2,h_3, h_4> 0$ and $0<\beta<\alpha -1$ such that for every $x>0$
\begin{equation}\label{eq: K pushes away}
   K(x)\geq h_1 x^{-\alpha}-h_2 x^{-\beta} -h_3-h_4x.
\end{equation}
Then there exists $D\subset \mathbb{R}^2$ such that diffusion $X=(X^{1,N}, X^{2,N})$ satisfies the assumptions in Section \ref{sec: assump and main thm}. Therefore the conclusion of Theorem \ref{thm: main theorem} holds for \eqref{eq: particle interactions 1}, \eqref{eq: particle interactions 2}.
\end{corollary}

\begin{proof}
    We set
    \begin{equation}
        D= \{x\in \mathbb{R}^2 \vert \; x^1<x^2 \},\;\;\;S=\{x\in \mathbb{R}^2 \vert \;x^1=x^2\} 
       ,
    \end{equation}
     so that standard calculus shows
    \begin{equation}
        \rho(x)=\inf_{y\in \mathbb{R}}\lvert (x_1,y_2)-(y,y)\rvert=\frac{x^{2}-x^1}{\sqrt{2}}.
    \end{equation}
One therefore has that $D$ is path connected $x_0\in D$, $D\cap S=\emptyset$ and $\partial D\subset S$, as required. To show \eqref{assmp: monotonicity} let us fix two elements $x,y \in D$ and denote by $b$ the drift coefficient of the entire system, so that by \eqref{eq: K monotonicity} and the fact $Q$ is Lipschitz one has
    \begin{align}
        \langle b(x)-b(y),x-y\rangle &= ((y^2-y^1)-(x^2-x^1))(K(y^2-y^1)-K(x^2-x^1))\nonumber \\
        &+\sum_{i=1,2}(Q(x^i)-Q(y^i))(x^i-y^i)\nonumber \\
        &\leq c\lvert x-y\rvert^2.
    \end{align}
To show \eqref{assmp: rho pushes away} one may use \eqref{eq: K pushes away} and the Lipschitz assumption on $Q$, in addition to the fact that $x^1<x^2$, to show that there exists $\tilde{h}_1, \tilde{h}_2, \tilde{h}_3, \tilde{h}_4>0$ such that for all $x\in D$ 
    \begin{align}
        &\langle b(x), \nabla \rho\rangle = 2^{1/2} K(x^2-x^1)+2^{-1/2}(Q(x^2)-Q(x^1)) \nonumber \\
        &\geq  \tilde{h}_1 \rho(x)^{-\alpha}- \tilde{h}_2\rho(x)^{\beta}- \tilde{h}_3- \tilde{h}_4\rho(x).
    \end{align}
Finally, to show \eqref{assmp: inverse poly Lipschitz b}, for $x,y\in D$ one has by \eqref{eq: K inverse poly Lipschitz} and the Lipschitz assumption on $Q$ that
    \begin{align}
   &\lvert b(x)-b(y)\rvert^2=4\lvert K(x^2-x^1)- K(y^2-y^1)\rvert^2+2\sum_{i=1,2} \lvert Q(x^i)-Q(y^i)\lvert^2\nonumber \\ 
   &\leq c(1+\rho(x)^{-l}+\rho(y)^{-l})^2\lvert x-y\rvert^2.
    \end{align}
\end{proof}

\appendix
\section{Appendix: Local It\^{o}'s Formula}
In this section we prove a local It\^{o}'s formula, which shall be crucial for the proofs of Lemmas \ref{lemma: strong cont moment bounds} and \ref{lemma: strong scheme inverse moment bounds}. A related Itô's formula was proven in \cite{GOBET2000167} for functions $f\in C^2(\bar{D})\cap C^0(\mathbb{R}^d)$ of processes that can escape the region $D$, which therefore necessitated extra terms to account for the behaviour on the boundary. In our case we consider a case in which the `regular' Itô's formula is recovered.
\begin{prop}\label{prop: local Ito}
    Let the assumptions given in Section \ref{sec: assump and main thm} hold, and let $(Y_t)_{t\geq 0}$ be a process satisfying
    \begin{equation}
        Y_t=\int^t_0 g_s ds+W_t,\;\;\; Y_0=y_0\in D,
    \end{equation}
    where $(W_t)_{t\geq 0}$ is a Wiener martingale defined on a filtered probability space $(\Omega, P, \mathcal{F}, (\mathcal{F}_t)_{t\geq 0})$, and where $g_t$ is an $\mathcal{F}_t$-measurable processes satisfying  
    \begin{equation}
\int^t_0 g_s ds<\infty,
    \end{equation}
    almost surely. Recall the definition of $K_\epsilon$ in \eqref{eq: K epsilon defn}, and suppose $\tau$ is a stopping time such that $Y_s\in K_\epsilon$ almost surely for $s\geq 0$ such that $0\leq s \leq \tau$ holds. Suppose $f\in C^2(D)$. Then one has 
    \begin{equation}
        f(Y_\tau)= f(y_0)+\int^\tau_0 (\langle \nabla f(Y_s), Y_s\rangle +\frac{1}{2}\Delta f(Y_s))ds+\int^\tau_0 \langle \nabla f(Y_s), dW_s\rangle.
    \end{equation}
\end{prop}
\begin{proof}
 We claim there exists a function $\phi\in C^\infty(\mathbb{R}^d)$ such that $\phi(x)=1$ for $x\in K_{3\epsilon/4}$ and $\phi(x)=0$ for $x\in \mathbb{R}^d\setminus K_{\epsilon/4}$. First let us show that 
 \begin{equation}\label{eq: sum identity 1}
     K_{3\epsilon/4}+B_{\epsilon/16}\subset K_{\epsilon/2}.
 \end{equation}
 To this end, let $x \in K_{3\epsilon/4}$, $y\in B_{\epsilon/16}$. Then for every $z\in S$ one has 
 \begin{equation}
     \lvert x+y-z\rvert \geq    \lvert x-z\rvert-\epsilon/16\geq 3\epsilon/4-\epsilon/16>\epsilon/2.
 \end{equation}
 Similarly one may show
 \begin{equation}\label{eq: sum identity 2}
     (\mathbb{R}^d\setminus K_{\epsilon/4}+B_{\epsilon/16})\cap K_{\epsilon/2}=\emptyset ,
 \end{equation}
since for every $x\in \mathbb{R}^d\setminus K_{\epsilon/4}$, $y\in B_{\epsilon/16}$ and $z\in S$ one has
 \begin{equation}
     \lvert x+y-z\rvert \leq    \lvert x-z\rvert+\epsilon/16<\epsilon/2.
 \end{equation}
 Now note furthermore that exists a function $\psi\in C^\infty_c(\mathbb{R}^d)$ such that $\int \psi(x)dx=1$ and $supp\; \psi \subset B_{\epsilon/16}$, see Section C.4 in \cite{evans1998partial}. Furthermore, by Theorem 6, C.4 in \cite{evans1998partial} one has that $\psi*1_{K_{\epsilon/2}}\in C^\infty(\mathbb{R}^d)$. Combining this with \eqref{eq: sum identity 1} and \eqref{eq: sum identity 2}, and noting that every $y\in K_{3\epsilon/4}+B_{\epsilon/16}$ satisfies $y\in B_{\epsilon/16}(x)$ for some $x\in K_{3\epsilon/4}$, one has for every $x\in K_{3\epsilon/4}$ that
 \begin{equation}
  \psi*1_{K_{\epsilon/2}}(x)=\int_{\mathbb{R}^d}   1_{K_{\epsilon/2}}(y)\psi(x-y)dy =\int_{B_{\epsilon/16}(x)}   1_{K_{\epsilon/2}}(y)\psi(x-y)dy=1.
 \end{equation}
Similarly, for every $x\in \mathbb{R}^d\setminus K_{\epsilon/4}$ one has $\psi*1_{K_{\epsilon/2}}(x)=0$. Therefore one may set $\phi=\psi*1_{K_{\epsilon/2}}$, and a result one has that $f \phi \in C^2(\mathbb{R}^d)$. Furthermore, since for every $x\in K_\epsilon$ one has that $\phi=1$ on a neighbourhood of $x$, it holds that for all multi-indices $\alpha$ of order $2$ one has $\partial ^\alpha (f \phi) = \partial ^\alpha f$ on $K_\epsilon$. Therefore, one may apply the classical It\^{o}'s formula to obtain an expression for  $f(Y_\tau)\phi(Y_\tau)$, at which point the result follows. 
\end{proof}

\subsection*{Acknowledgements}
Thank you to István Gyöngy, Jonathan C. Mattingly and Emmanuel Gobet for fruitful discussion, and for the suggestion of several related works in the literature.
\bibliographystyle{plain}
\bibliography{ref}
\end{document}